\documentclass{amsart}
\usepackage{amsmath,amssymb,amsthm}
\usepackage{enumerate}
\usepackage{hyperref}
\usepackage{mathtools,pifont}
\usepackage{tikz-cd}
\usepackage{tikz}
\usepackage{placeins}
\usepackage{xcolor}
\usepackage{mathabx}
\usepackage{float}
\usepackage[margin=1in]{geometry}

\newtheorem{theorem}{Theorem}[section]

\newtheorem{proposition}[theorem]{Proposition}

\numberwithin{equation}{section}

\theoremstyle{definition}

\newtheorem{example}[theorem]{Example}

\theoremstyle{remark}
\newtheorem{remark}[theorem]{Remark}

\title[New Type degenerate Simsek numbers and related aspects]{New Type degenerate Simsek numbers and related aspects}

\author{Lahcen Oussi}
\address{\parbox{\linewidth}{Department of Analysis and Stochastic Processes\\Faculty of Pure and Applied Mathematics
\\Wroclaw University of Science and Technology \\
Wybrze\.{z}e Wyspia\'nskiego 27, 50-370, Wroclaw, Poland}}
\email{lahcen.oussi@pwr.edu.pl, oussimaths@gmail.com}

\begin{document}

\begin{abstract}
In this paper, we introduce a new type degenerate Simsek numbers and their generating function, which are different from degenerate Simsek number studied so far. We establish the explicit formula, recurrence relation and other identities for these numbers. We also derive several interesting expressions and relations between these numbers and certain other special numbers in the literature. In addition, several numerical examples and graphical illustrations are provided to support the theoretical results and to illustrate the behavior of the introduced numbers.
\end{abstract}


\maketitle

\begin{section}
{Introduction and Preliminaries}
\end{section}

Special numbers and polynomials play  crucial rule to solve explicitly several problems in many scientific areas, such as in mathematical physics, biology, quantum mechanics, discrete mathematics, probability and statistics. The very well known special numbers are Stirling numbers of the first and second kind introduced by James Stirling in 1730.   

Stirling numbers of the second kind $S_{1}(n,k)$ count the number of partitions of a set of
size $n$ into $k$ disjoint, non-empty subsets. They appear as coefficients in the following expansion:
\begin{equation}\label{SNSK}
    x^n=\sum_{k=0}^{n}S_{2}(n,k)(x)_{k},
\end{equation}
where $$(x)_k=\begin{cases}
   1 & \text{ if } k=0,  \\[5pt]
    x(x-1)(x-2)\cdots (x-k+1) & \text{ if } k\geq 1,
  \end{cases}$$
denotes the falling factorial (cf.\ \cite{LC1974,SR1983}). The generating function of $S_{2}(n,k)$ is given by
\begin{equation}\label{GFSNSK}
 \frac{1}{k!}(e^t-1)^k=\sum_{n=k}^{\infty}S_{2}(n,k)\frac{t^n}{n!}.  
\end{equation}
Stirling numbers of the first kind, denoted by $S_{1}(n,k)$, are given by the following:
\begin{equation}\label{SNFK}
    (x)_n=\sum_{k=0}^{n}S_{1}(n,k)x^k,
\end{equation}
with its generating function
\begin{equation}\label{GFSNFK}
   \frac{(\log(1+t))^k}{k!}=\sum_{n=k}^{\infty}S_{1}(n,k)\frac{t^n}{n!}, 
\end{equation}
(cf.\ \cite{LC1974,SR1983}).

The modifications and degenerate versions of these numbers were studied by several researchers, see for instance \cite{LO2023, ABYSHMS2024, FTH1985, LCHPJSS1998, DSKTK2017, TKDSKHYKJW2019, IKYS2021, TMMS2012, LO2020, LO2022, MDS2028, YS2013, HKK2022} and the references therein.

For a nonzero $\alpha\in\mathbb{R}$ (or $\mathbb{C}$), the degenerate exponential function $e_{\alpha}^{x}(t)$ defined as follows
\begin{equation}\label{degexpfun}
    e_{\alpha}^{x}(t)=(1+\alpha t)^{\frac{x}{\alpha}} \text{ and } e_{\alpha}(t)=(1+\alpha t)^{\frac{1}{\alpha}},\quad \text{(see,~\cite{LC1979, TKDSK2017, TKDSKHIK2017, DSKTK2020, DSKTKHYKHL2020})}.
\end{equation}
Using the Taylor expansion, the degenerate exponential function can be expressed as follows:
\begin{equation}\label{tdegexpfun}
 e_{\alpha}^{x}(t)=\sum_{n=0}^{\infty}(x)_{n,\alpha}\frac{t^n}{n!},   
\end{equation}

where 
$$(x)_{n, \alpha}=\begin{cases}
     1 & \text{ if } n=0, \\[6pt]
    x(x-\alpha)(x-2\alpha)\cdots (x-(n-1)\alpha) & \text{ if } n\geq 1,
  \end{cases}$$
is called the degenerate falling factorial. It follows that 
$$\lim_{\alpha\rightarrow{ 0}}e_{\alpha}^{x}(t)=\sum_{}\frac{x^nt^n}{n!}=e^{xt}.$$
The degenerate Stirling numbers of the second kind $S_{2, \alpha}(n, l)$ are given as follows \cite{TK2017, DSKTK2020, DSKTKHYKHL2020, TKDSKHKK2022}:
\begin{equation}\label{dsnsc}
  (x)_{n, \alpha}=\sum_{l=0}^{n}S_{2, \alpha}(n, l)(x)_{l}. 
\end{equation}
Moreover, as an inversion formula of \eqref{dsnsc}, the degenerate Stirling numbers of the fist kind $S_{1, \alpha}(n, l)$ are defined as follows \cite{TK2017, DSKTK2020, DSKTKHYKHL2020}:
\begin{equation}\label{dsnfc}
  (x)_{n}=\sum_{l=0}^{n}S_{1, \alpha}(n, l)(x)_{l.\alpha}. 
\end{equation}
Note that 
$$\lim_{\alpha\rightarrow 0}S_{1, \alpha}(n, l)=S_{1}(n, l) \text{ and } \lim_{\alpha\rightarrow 0}S_{2, \alpha}(n, l)=S_{2}(n, l)$$
In addition, a new type degenerate Stirling numbers of the second kind $S_{2}^*(n,k|\alpha)$ were introduced in \cite{HKK2022} as 
\begin{equation}\label{ntSnHKK}
\frac{1}{k!}(e^t-1)_{k,\alpha}=\sum_{n=k}^{\infty}S_{2}^*(n,k|\alpha)\frac{t^n}{n!} \quad \text{ and } \quad S_{2}^*(n,0|\alpha)=0, \quad(n\geq 0).
\end{equation}
They reduced to Stirling numbers of the second kind~\eqref{GFSNSK}.

In 2018, Simsek \cite{YS1018} introduced new families of special numbers $y_{1}(n,k;\lambda)$, for computing negative order Euler numbers and related numbers and polynomials. Simsek numbers $y_{1}(n,k;\lambda)$ are defined  by means of generating function 
\begin{equation}\label{Snumbers}
  \frac{\Big(\lambda e^t+1\Big)^k}{k!}=\sum_{n=0}^{\infty}y_{1}(n,k;\lambda)\frac{t^n}{n!}.  
\end{equation}
They can be expressed explicitly by the following identity
\begin{equation}\label{SnumbersExp}
   y_{1}(n,k;\lambda)=\frac{1}{k!}\sum_{j=0}^{k}\binom{k}{j} j^n\lambda^j.
\end{equation}
Simsek numbers $y_{1}(n,k;\lambda)$ are related to the many well-known special numbers in the literature, among them, Bernoulli
numbers, Fibonacci numbers, Stirling numbers of the second kind, Lucas numbers and the central numbers. 

Recently, the author introduced a degenerate version of Simsek numbers $y_{1}(n,k;\lambda|\alpha)$ in \cite{LO2023}, by the mean of the generating function 
\begin{equation}\label{DSnumbers}
 \frac{\Big(\lambda e^{\frac{\log(1+\alpha t)}{\alpha}}+1\Big)^k}{k!}=\sum_{n=0}^{\infty}y_{1}(n,k;\lambda|\alpha)\frac{t^n}{n!}, 
\end{equation}
with the explicit formulas
\begin{equation}\label{expformdS1}
   y_{1}(n,k;\lambda|\alpha)=\frac{1}{k!}\sum_{m=0}^{n}\sum_{j=0}^{k} \binom{k}{j}j^m\lambda^j\alpha^{n-m}S_1(n,m)
\end{equation}
and
\begin{equation}\label{expformdS2}
   y_{1}(n,k;\lambda|\alpha)=\frac{1}{k!}\sum_{j=0}^{k} \binom{k}{j}\lambda^j\alpha^{n}\Big(\frac{j}{\alpha}\Big)_{_{n}}.
\end{equation}
For more details about the numbers $y_{1}(n,k;\lambda|\alpha)$ and related aspects, we refer the readers to \cite{LO2023}.

Kucukoglu \cite{IK2025, IK2023}, investigated the $q$-combinatorial Simsek numbers  $y_{1,q}(n,k;\lambda)$ and polynomials  $y_{1,q}(x;n,k;\lambda)$ of the first kind, respectively, as follows
$$y_{1,q}(n,k;\lambda)=\frac{1}{[k]_q!}\sum_{j=0}^{k}q^{\binom{j}{2}}\begin{bmatrix}
k \\
j 
\end{bmatrix}_{q}[j]_{q}^{n}\lambda^j$$
and 
$$y_{1,q}(x;n,k;\lambda)=\frac{1}{[k]_q!}\sum_{j=0}^{k}q^{\binom{j}{2}}\begin{bmatrix}
k \\
j 
\end{bmatrix}_{q}[x+j]_{q}^{n}\lambda^j.$$
In \cite{IKBSYS2019}, Kucukoglu et al.~defined the higher order expansion of the Simsek numbers and polynomials. Moreover, Kilar \cite{NK2024}, introduced degenerate Simsek-type numbers and polynomials of higher order. 
 
In this paper, we establish a new type degenerate Simsek numbers, which are different from degenerate Simsek numbers \eqref{DSnumbers}.
\begin{section}
{New type degenerate Simsek numbers and related formulas}
\end{section}
In this section, we define a new type degenerate Simsek numbers and we derive some related formulas including generating function, derivative formulas, recurrence formulas, and integral formulas. Moreover, we establish also a relation between these numbers and certain special numbers.

We define a new type degenerate Simsek numbers $y^{*}_{1,\alpha}(n,k;\lambda)$ by means of generating function: 
 \begin{equation}\label{newgSn}
F_{k}(t; \alpha, \lambda):=\frac{(\lambda e^t+1)_{k,\alpha}}{k!}=\sum_{n=0}^{\infty}y_{1,\alpha}^{*}(n,k;\lambda)\frac{t^n}{n!}.  
 \end{equation}
It follows that for $\alpha=0$, we obtain $y_{1, \alpha}^{*}(n,k;\lambda)=y_{1}(n,k;\lambda)$ Simsek numbers \eqref{Snumbers}.  Moreover, the function $F_{k}(t;\alpha,k)$ can be expressed as follows 
\begin{equation}\label{BerPoly}
F_{k}(t;\alpha,\lambda)=\frac{\alpha^k}{k!}\bigg(\frac{\lambda e^t+1}{\alpha}\bigg)_{k}=\frac{\alpha^k}{k!}B_{k}^{(k+1)}\left(\frac{\lambda e^t+1}{\alpha}+1\right),
\end{equation}
where $B_{k}^{(n)}(x)$ denotes Bernoulli's polynomials of order $n$ given by 
\begin{equation}
\frac{t^ne^{xt}}{(e^t-1)^n}=\sum_{k=0}^{\infty}B_{k}^{(n)}(x)\frac{t^k}{k!}
\end{equation}
and satisfies the recurrence relation
\begin{equation}\label{BerPolyrec}
B_{k}^{(n+1)}(x)=\left(1-\frac{k}{n}\right)B_{k}^{(n)}(x)+k\left(\frac{x}{n}-1\right)B_{k-1}^{(n)}(x),
\end{equation}
(cf.~\cite{LLM1980}).

By using~\eqref{newgSn}, we derive the following functional equation
\begin{equation}
\alpha^k\left(\frac{\lambda e^t}{\alpha}\right)_{k}=(\lambda e^t)_{k, \alpha}=\sum_{\ell=0}^{k}(-1)_{k-\ell,\alpha}\binom{k}{\ell}\ell! F_{\ell}(t;\alpha,\lambda).
\end{equation}
Combining~\eqref{newgSn} and~\eqref{SNFK} with the above equation, yields
\begin{equation}\label{eq:proveSY}
\sum_{n=0}^{\infty}\left(\sum_{j=0}^{k}\alpha^{k-j}S_{1}(k,j)\lambda^{j}j^n\right)\frac{t^n}{n!}=\sum_{n=0}^{\infty}\left(\sum_{\ell=0}^{k}(-1)_{k-\ell, \alpha}\ell!\binom{k}{\ell}y_{1,\alpha}^{*}(n,\ell;\lambda)\right)\frac{t^n}{n!}. 
\end{equation}
Comparing the coefficients of $\frac{t^n}{n1}$ on both sides of~\eqref{eq:proveSY}, we obtain the
following relation between the numbers $y_{1,\alpha}^{*}(n,k;\lambda)$ and Stirling numbers of the first kind $S_{1}(n, k)$:
\begin{theorem}
\begin{equation}
    \sum_{j=0}^{k}\alpha^{k-j}S_{1}(k,j)\lambda^jj^n=\sum_{\ell=0}^{k}(-1)_{k-\ell, \alpha}\ell!\binom{k}{\ell}y_{1,\alpha}^{*}(n,\ell;\lambda).
\end{equation}
\end{theorem}
\begin{remark}
It follows that for $\alpha=0$, we obtain 
\begin{equation*}
  \lambda^{k}k^n=\sum_{\ell=0}^{k}(-1)^{k-\ell}\binom{k}{\ell}\ell!y_{1}(n,\ell;\lambda)
\end{equation*}
the result given in~\cite{YS1018}.
\end{remark}
\begin{theorem}\label{thm:formntdSn}
The numbers $y_{1, \alpha}^{*}(n,k;\lambda)$ can be expressed explicitly  as follows:
 \begin{equation}\label{expforntsn}
 y_{1,\alpha}^{*}(n,k;\lambda)=\frac{1}{k!}\sum_{\ell=0}^{k}\sum_{j=0}^{\ell}\binom{\ell}{j}\alpha^{k-\ell}S_{1}(k,\ell)\lambda^{j}j^{n}
    \end{equation}
or
\begin{equation}\label{expforntsn2}
y_{1, \alpha}^{*}(n,k;\lambda)=\frac{1}{k!}\sum_{\ell=0}^{k}\sum_{j=0}^{\ell}\binom{k}{\ell}\alpha^{\ell-j}\lambda^{j}j^{n}(1)_{k-\ell, \alpha}S_{1}(\ell,j).
\end{equation}
\end{theorem}
\begin{proof}
Using \eqref{SNFK}, we have 
\begin{align*}
\frac{(\lambda e^t+1)_{k,\alpha}}{k!}&=\frac{\alpha^{k}}{k!}\Big(\frac{\lambda e^t+1}{\alpha}\Big)_{k}\\
&=\frac{1}{k!}\sum_{\ell=0}^{k}\alpha^{k-\ell}S_{1}(k,\ell)(\lambda e^t+1)^{\ell}\\
&=\frac{1}{k!}\sum_{\ell=0}^{k}\alpha^{k-\ell}S_{1}(k,\ell)\sum_{j=0}^{\ell}\binom{\ell}{j}\lambda^{j}e^{jt}\\
&=\sum_{n=0}^{\infty}\Bigg(\frac{1}{k!}\sum_{\ell=0}^{k}\sum_{j=0}^{\ell}\alpha^{k-\ell}S_{1}(k,\ell)\binom{\ell}{j}\lambda^{j}j^n\Bigg)\frac{t^n}{n!}.
\end{align*}
Hence, by \eqref{newgSn}, we obtain the desired result.

To prove Eq.~\eqref{expforntsn2}, we use the fact that $(x+y)_{k, \alpha}=\sum\limits_{j=0}^{k}\binom{k}{j}(x)_{j,\alpha}(y)_{k-j, \alpha}$. Then,
\begin{align*}
\frac{\big(\lambda e^t+1\big)_{k, \alpha}}{k!}&=\frac{1}{k!}\sum_{\ell=0}^{k}\binom{k}{\ell}(\lambda e^t)_{\ell, \alpha}(1)_{k-\ell, \alpha}\\
&=\frac{1}{k!}\sum_{\ell=0}^{k}\binom{k}{\ell}\alpha^{\ell}\Big(\frac{\lambda e^t}{\alpha}\Big)_{\ell}(1)_{k-\ell, \alpha}\\
&=\sum_{n=0}^{\infty}\left(\frac{1}{k!}\sum_{\ell=0}^{k}\sum_{j=0}^{\ell}\binom{k}{\ell}\alpha^{\ell-j}\lambda^{j}j^{n}(1)_{k-\ell, \alpha}S_{1}(\ell,j)\right)\frac{t^n}{n!}.
\end{align*}
Hence the result follows.
\end{proof}
On the other hand, using the fact that $F_{k}(t;\alpha,\lambda)=\frac{\alpha^k}{k!}\Big(\frac{\lambda e^t+1}{\alpha}\Big)_{k}$ and the  identity 
\begin{equation}
(x)_{k}=\sum_{j=0}^{k}\binom{k}{j}\frac{j}{k}x^j B_{k-j}^{(k)},
\end{equation}
where $B_{n}^{(k)}$ denotes Bernoulli's numbers of order $k$, defined by 
\begin{equation}
\Big(\frac{t}{e^t-1}\Big)^k=\sum_{k=0}^{\infty}B_{n}^{(k)}\frac{t^n}{n!}
\end{equation} 
(cf.~\cite{LLM1980}), we obtain the following proposition:
\begin{proposition}
The numbers $y_{1,\alpha}^{*}(n,k;\lambda)$ are given, in relation with Bernoulli's numbers $B_{k}^{(n)}$ of order $n$, by the following 
\begin{equation}\label{expforntsn3}
y_{1,\alpha}^{*}(n,k;\lambda)=\frac{1}{k!}\sum_{j=0}^{k}\alpha^{k-j}\binom{k}{j}\frac{j}{k}B_{k-j}^{(k)}\sum_{\ell=0}^{j}\binom{j}{\ell}\lambda^{\ell}\ell^{n}.
\end{equation}
\end{proposition}
\begin{remark}
Substituting  $\alpha=0$ in Equations~\eqref{expforntsn3},~\eqref{expforntsn} and~\eqref{expforntsn2}, we obtain
$$y_{1, 0}^{*}(n,k;\lambda)=y_{1}(n,k;\lambda)=\frac{1}{k!}\sum_{j=0}^{k}\binom{k}{j}j^n\lambda^j$$
Simsek numbers \eqref{SnumbersExp}.
\end{remark}
\begin{example} Some special cases of the numbers $y_{1, \alpha}^{*}(n,k;\lambda)$ are given as follows:
\begin{itemize}
  \item $y_{1,\alpha}^{*}(n,0;\lambda)=\delta_{n0}$
  \item $y_{1,\alpha}^{*}(n,1;\lambda)=\delta_{n0}+\lambda$
  \item $y_{1,\alpha}^{*}(0,k;\lambda)=\frac{\alpha^k}{k!}\sum\limits_{\ell=0}^{k}S_{1}(k,\ell)\left(\frac{\lambda+1}{\alpha}\right)^{\ell}$
\end{itemize}
\end{example}

In the following theorem, we give a relation between the numbers $y_{1, \alpha}^{*}(n, k; \lambda)$, Simsek numbers $y_{1}(n, k; \lambda)$, Stirling numbers of the first kind $S_{1}(n, k)$ and the degenerate Stirling numbers of the second kind $S_{2.\alpha}(n, k)$.
\begin{theorem}
The numbers $y_{1, \alpha}^{*}(n,k;\lambda)$ can be expressed as follows:
    $$y_{1,\alpha}^{*}(n,k;\lambda)=\frac{1}{k!}\sum_{\ell=0}^{k}\sum_{j=0}^{\ell}S_{2,\alpha}(k,\ell)S_{1}(\ell,j)j!y_{1}(n,j;\lambda).$$
\end{theorem}
\begin{proof}
Using \eqref{SNFK}, \eqref{Snumbers}, \eqref{newgSn} and \eqref{dsnsc}, we have
\begin{align}
    \sum_{n=0}^{\infty}y_{1, \alpha}^{*}(n,k;\lambda)\frac{t^n}{n!}&=\frac{1}{k!}\sum_{\ell=0}^{k}S_{2,\alpha}(k,\ell)(\lambda e^t+1)_\ell\\
    &=\frac{1}{k!}\sum_{\ell=0}^{k}S_{2,\alpha}(k,\ell)\sum_{j=0}^{\ell}S_{1}(\ell,j)(\lambda e^t+1)^j\\
    &=\frac{1}{k!}\sum_{\ell=0}^{k}S_{2,\alpha}(k,\ell)\sum_{j=0}^{\ell}S_{1}(\ell,j)j!\sum_{n=0}^{\infty}y_{1}(n,j;\lambda)\frac{t^n}{n!}.
\end{align}
Comparing the coefficients of both sides of the above identity, we obtain the desired result.
\end{proof} 
On the other hand, the numbers $y_{1,\alpha}^*(n,k;\lambda)$ are related to the degenerate Stirling numbers of the second kind $S_{2}^*(n,k| \alpha)$ as follows:
\begin{theorem}
\begin{equation}\label{ntsnntsn}
  y_{1,\alpha}^*(n,k;\lambda)=\frac{1}{k!}\sum_{j=0}^{k}\binom{k}{j}j!\lambda^j(\lambda+1)_{k-j,\alpha}S_{2}^*\left(n,k|\frac{\alpha}{\lambda}\right).
\end{equation}
\end{theorem}
\begin{proof} We have 
\begin{align*}
  \frac{\left(\lambda e^t+1\right)_{k,\alpha}}{k!} = & \frac{1}{k!}\left(\lambda(e^t-1)+\lambda+1\right)_{k,\alpha} \\
  = & \frac{1}{k!}\sum_{j=0}^{k}\binom{k}{j}\left(\lambda(e^t-1)\right)_{j,\alpha}(\lambda+1)_{k-j,\alpha}\\
  = &  \frac{1}{k!}\sum_{j=0}^{k}\binom{k}{j}\lambda^j (e^t-1)_{j,\frac{\alpha}{\lambda}}(\lambda+1)_{k-j,\alpha}\\
  = & \frac{1}{k!}\sum_{j=0}^{k}\binom{k}{j}\lambda^j j!\lambda^j (\lambda+1)_{k-j,\alpha}\sum_{n\geq 0} S_{2}^{*}\left(n,j|\frac{\alpha}{\lambda}\right)\frac{t^n}{n!}
\end{align*}
Using~\eqref{newgSn} and equating the coefficients in both sides of the above identity finishes the proof.
\end{proof}
\begin{theorem}
The numbers $y_{1,\alpha}^{*}(n,k;\lambda)$ satisfy the following recurrence formula with respect to $k$
\begin{equation}\label{rfngsmk}
 y_{1,\alpha}^{*}(n,k+1;\lambda)=\frac{1}{k+1}\bigg(\lambda\sum_{\ell=0}^{n}\binom{n}{\ell}y_{1,\alpha}^{*}(\ell,k;\lambda)+(1-k\alpha)y_{1,\alpha}^{*}(n,k;\lambda)\bigg).   
\end{equation}
\end{theorem}    
\begin{proof}
From \eqref{newgSn}, we have 
\begin{equation}\label{rel:fkfkmius}
F_{k}(t;\alpha,\lambda)=\left(\frac{\lambda e^t+\alpha-k\alpha+1}{k}\right)F_{k-1}(t;\alpha,\lambda).
\end{equation}
Hence,
\begin{align}
    \sum_{n=0}^{\infty}y_{1,\alpha}^{*}(n,k+1;\lambda)\frac{t^n}{n!}
    &=\frac{\lambda}{k+1}\sum_{m=0}^{\infty}\frac{t^m}{m!}\sum_{n=0}^{\infty}y_{1,\alpha}^{*}(n,k;\lambda)\frac{t^n}{n!}+\frac{1-k\alpha}{k+1}\sum_{n=0}^{\infty}y_{1,\alpha}^{*}(n,k;\lambda)\frac{t^n}{n!}\\
    &=\frac{1}{k+1}\sum_{n=0}^{\infty}\bigg(\lambda\sum_{\ell=0}^{n}y_{1,\alpha}^{*}(\ell,k;\lambda)\binom{n}{\ell}+(1-k\alpha)y_{1,\alpha}^{*}(n,k;\lambda)\bigg)\frac{t^n}{n!}.\label{lasteqrf}
\end{align}
By comparing the coefficients of both sides of \eqref{lasteqrf}, we have \eqref{rfngsmk}.
\end{proof}

Taking the derivative with respect to $t$ of both sides of \eqref{rel:fkfkmius}, we obtain
\begin{equation*}
F'_{k}(t;\alpha,\lambda)=\frac{\lambda e^t}{k}\left(F_{k-1}(t;\alpha,\lambda)+F'_{k-1}(t;\alpha.\lambda)\right)+\left(\frac{1+\alpha-k\alpha}{k}\right)F'_{k-1}(t;\alpha.\lambda).
\end{equation*}
Then, using \eqref{newgSn}, the above identity becomes
\begin{equation*} 
\begin{split}
\sum_{n=0}^{\infty}y_{1,\alpha}^{*}(n+1,k;\lambda)\frac{t^n}{n!}& = \sum_{n=0}^{\infty}\Bigg[\frac{\lambda}{k}\sum_{j=0}^{n}\binom{n}{j}\left(y_{1,\alpha}^{*}(j,k-1;\lambda)+y_{1,\alpha}^{*}(j+1,k-1;\lambda)\right)\\
 &+\left(\frac{1+\alpha-k\alpha}{k}\right)y_{1,\alpha}^{*}(n+1,k-1;\lambda)\Bigg]\frac{t^n}{n!}.
\end{split}
\end{equation*}
Therefore, equating the coefficients of $\frac{t^n}{n!}$, gives another recursive formula for the numbers $y_{1,\alpha}^{*}(n,k;\lambda)$ in the following theorem.
\begin{theorem}\label{thmrecforn}
The numbers $y_{1,\alpha}^{*}(n,k;\lambda)$ satisfy the following recurrence formula
\begin{equation*}
\begin{split}
 y_{1,\alpha}^{*}(n+1,k;\lambda)&=\frac{\lambda}{k}\sum_{j=0}^{n}\binom{n}{j}\bigg(y_{1,\alpha}^{*}(j,k-1;\lambda)+y_{1,\alpha}^{*}(j+1,k-1;\lambda)\bigg)\\
 &\quad\qquad +\left(\frac{1+\alpha-k\alpha}{k}\right)y_{1,\alpha}^{*}(n+1,k-1;\lambda)\text{,} \qquad \text{ for } k\geq 1   
 \end{split}
\end{equation*}
and $$y_{1,\alpha}^{*}(n,0;\lambda)=\delta_{n,0}.$$
\end{theorem} 
\begin{remark}
Note that, one can also prove Theorem \ref{thmrecforn}, using \eqref{BerPoly}, \eqref{BerPolyrec} and the derivative of Bernoulli's polynomials of order $k$
\begin{equation}
\frac{d}{dt}B_{n}^{(k)}(t)=nB_{n-1}^{(k)}(t)
\end{equation}
(cf.~\cite{LLM1980}) 
\end{remark}
We define the generating function $\phi^{*}_{n}(x|\alpha,\lambda)$ of the numbers $y_{1,\alpha}^{*}(n,k;\lambda)$ as follows:
\begin{equation}\label{gfndsn}
 \phi^{*}_{n}(x|\alpha,\lambda):=\sum\limits_{k=0}^{\infty}y_{1,\alpha}^{*}(n,k;\lambda)x^k.  
\end{equation}
In the further, we derive some formulas related to the generating function $\phi^{*}_{n}(x|\alpha,\lambda)$, such as exponential generating function, recurrence formula, integral formula, derivative formula and relation with some well-know numbers.
\begin{theorem}\label{gfphi}
For a  non-negative integer $n$, we have the following exponential generating function of $\phi^{*}_{n}(x|\alpha,\lambda)$:
   \begin{equation}\label{exgefphi}
    \sum_{n=0}^{\infty}\phi_{n}^{*}(x|\alpha,\lambda)\frac{t^n}{n!}=e_{\alpha}^{\lambda e^t+1}(x)   
   \end{equation}
\end{theorem}
\begin{proof} we have

    \begin{align}
        \sum_{n=0}^{\infty}\phi_{n}^{*}(x|\alpha,\lambda)\frac{t^n}{n!}        &=\sum_{n=0}^{\infty}\Big(\sum_{k=0}^{\infty}y_{1,\alpha}^{*}(n,k;\lambda)x^k\Big)\frac{t^n}{n!}\\
        &=\sum_{k=0}^{\infty}\Big(\sum_{n=0}^{\infty}y_{1,\alpha}^{*}(n,k;\lambda)\frac{t^n}{n!}\Big)x^k\\
        &=\sum_{k=0}^{\infty}(\lambda e^t+1)_{k,\alpha}\frac{x^k}{k!}\\
        &=e_{\alpha}^{\lambda e^t+1}(x),
    \end{align}
where in the third equation we used \eqref{newgSn} and in the last equation we used \eqref{degexpfun}.
\end{proof}
The following result, describe a relation between the generating functions $\phi_{n}^{*}(x|\alpha,\lambda)$ and $\phi_{n}^{*}(x|0,\lambda)$ of the numbers $y_{1,\alpha}^{*}(n,k;\lambda)$ and Simsek numbers $y_{1}(n,k;\lambda)$, respectively.
\begin{theorem}
The generating function $\phi_{n}^{*}(x|\alpha,\lambda)$ can be expressed as follows:
    $$\phi_{n}^{*}(x|\alpha,\lambda)=\sum_{k=0}^{\infty}\bigg(\frac{\log(1+\alpha x)}{\alpha}\bigg)^{k}y_{1}(n,k;\lambda)=\phi_{n}^{*}\left(\frac{\log(1+\alpha x)}{\alpha}|0,\lambda\right),$$
    where $\phi_{n}^{*}(x|0,\lambda)$ is the generating function of Simsek numbers $y_{1}(n,k;\lambda)$.
\end{theorem}
\begin{proof}
By Theorem~\ref{gfphi}, we have
\begin{align}
    \sum_{n=0}^{\infty}\phi_{n}^{*}(x|\alpha,\lambda)\frac{t^n}{n!}&=e_{\alpha}^{\lambda e^t+1}(x)\\
    &=(1+\alpha x)^{\frac{\lambda e^t+1}{\alpha}}\\
    &=e^{\frac{\lambda e^t+1}{\alpha}\log(1+\alpha x)}\\
    &=\sum_{k=0}^{\infty}\frac{1}{\alpha^k}\big(\log(1+\alpha x)\big)^k\frac{(\lambda e^t+1)^k}{k!}\\
    &=\sum_{n=0}^{\infty}\sum_{k=0}^{\infty}\Big(\frac{\log(1+\alpha x)}{\alpha}\Big)^{k}y_{1}(n,k;\lambda)\frac{t^n}{n!}\label{lasteqgf}.
\end{align}
Hence, the result follows by comparing the coefficients of both sides of \eqref{lasteqgf}.
\end{proof}
\begin{theorem}
For $n\geq 0$, the generating function $\phi_{n}^*(x,\alpha|\lambda)$ satisfies the following recursive formula:
    \begin{equation}\label{recforphi}
        \phi_{n+1}^*(x|\alpha,\lambda)=\frac{\lambda}{\alpha}\log(1+\alpha x)\sum_{\ell=0}^{n}\binom{n}{\ell}\phi^{*}_\ell(x|\alpha,\lambda)
    \end{equation}
\end{theorem}
\begin{proof}
 Taking the derivative with respect to $t$ in both sides of \eqref{exgefphi}, we obtain 
\begin{equation}\label{dertphi2}
 \frac{d}{dt}\sum_{n=0}^{\infty}\phi^{*}_{n}(x|\alpha,\lambda)\frac{t^n}{n!}=\sum_{n=0}^{\infty}\phi^{*}_{n+1}(x|\alpha,\lambda)\frac{t^n}{n!},   
\end{equation}
and 
\begin{align}
        \frac{d}{dt}\left(e_{\alpha}^{\lambda e^t+1}(x)\right)&=\frac{\lambda}{\alpha}e^t\log(1+\alpha x)e_{\alpha}^{\lambda e^t+1}(x)\\
        &=\frac{\lambda}{\alpha}\log(1+\alpha x)\sum_{j=0}^{\infty}\frac{t^j}{j!}\sum_{\ell=0}^{\infty}\phi_{\ell}^*(x|\alpha,\lambda)\frac{t^\ell}{\ell!}\\
        &=\frac{\lambda}{\alpha}\log(1+\alpha x)\sum_{n=0}^{\infty}\sum_{\ell=0}^{n}\binom{n}{\ell}\phi_{\ell}^*(x|\alpha, \lambda)\frac{t^n}{n!}\label{dertphi1}.
    \end{align}
Equating the coefficients of $\frac{t^n}{n!}$ in \eqref{dertphi1} and \eqref{dertphi2} gives the desired result.
\end{proof}
 We express the generating function $\phi^{*}_{n}(x|\alpha,\lambda)$ in relation with the degenerate first kind Apostol-Euler numbers $E_{n}^{(k)}(\lambda|\alpha)$ of order $k$ defined in \cite{LO2023} as follows 
\begin{equation}\label{dfkAEn}
 \bigg(\frac{2}{\lambda e^{\frac{\log(1+\alpha t)}{\alpha}}+1}\bigg)^k=\sum_{n=0}^{\infty}E_{n}^{(k)}(\lambda|\alpha)\frac{t^n}{n!}.
\end{equation}  
\begin{theorem}
For a non-negative integer $n$, we have 
\begin{equation}
    \sum\limits_{m=0}^{n}\binom{n}{m}E_{n-m}^{(1)}(\lambda|0)\frac{d}{dx}\phi^{*}_{m}(x|\alpha,\lambda)=\frac{2}{1+\alpha x}\phi^{*}_{n}(x|\alpha,\lambda).
  \end{equation}  
\end{theorem}
\begin{proof}
Multiplying by $\frac{2}{1+\lambda e^{t}}$ in both sides of \eqref{derxphiarr1}, we obtain
   \begin{equation}\label{derxphipro}
     \frac{2}{1+\lambda e^{t}}\sum_{n=0}^{\infty}\frac{d}{dx}\phi^{*}_{n}(x|\alpha,\lambda)\frac{t^n}{n!}=\frac{2}{1+\alpha x}e_{\alpha}^{\lambda e^{t}+1}(x) 
   \end{equation} 
The right-hand side of \eqref{derxphipro} becomes
\begin{equation}\label{rhandside}
   \frac{2}{1+\alpha x}\sum_{n=0}^{\infty}\phi^{*}_{n}(x|\alpha,\lambda)\frac{t^{n}}{n!}. 
\end{equation}
Using~\eqref{dfkAEn}, the left-hand side of \eqref{derxphipro} is written as follows
\begin{align}
 \frac{2}{1+\lambda e^{t}}\sum_{n=0}^{\infty}\frac{d}{dx}\phi^{*}_{n}(x|\alpha,\lambda)\frac{t^n}{n!}&=\sum_{j=0}^{\infty}E^{(1)}_{j}(\lambda|0)\frac{t^{j}}{j!}\sum_{m=0}^{\infty}\frac{d}{dx}\phi^{*}_{m}(x|\alpha,\lambda)\frac{t^m}{m!}\\
 &=\sum_{n=0}^{\infty}\sum_{m=0}^{n}\binom{n}{m}E^{(1)}_{n-m}(\lambda|0)\frac{d}{dx}\phi_{m}^{*}(x|\alpha,\lambda)\frac{t^{n}}{n!}\label{lhandside}.
\end{align}
Hence, the result follows from \eqref{rhandside} and \eqref{lhandside}.
\end{proof}
Using Theorem \ref{gfphi} and the fat that   
\begin{equation}\label{derxphi1}
      \frac{d}{dx}(e_{\alpha}^{\lambda e^t+1}(x))=\frac{\lambda e^t+1}{1+\alpha x}e_{\alpha}^{\lambda e^t+1}(x),
\end{equation}
we get
\begin{align}
    \sum_{n=0}^{\infty}\frac{d}{dx}\phi^{*}_{n}(x|\alpha,\lambda)\frac{t^n}{n!}\label{derxphiarr1}&=\frac{1}{1+\alpha x}\bigg((1+ \lambda e^{t})e_{\alpha}^{e^{\lambda t+1}}(x)\bigg)\\ 
    &=\frac{1}{1+\alpha x}\bigg((1+ \lambda e^{t})\sum_{\ell=0}^{\infty}\phi^{*}_{\ell}(x|\alpha,\lambda)\frac{t^\ell}{\ell!}\bigg)\\ 
    &=\frac{1}{1+\alpha x}\sum_{n=0}^{\infty}\bigg(\lambda\sum_{\ell=0}^{n}\binom{n}{\ell}\phi^{*}_{\ell}(x|\alpha,\lambda)+\phi^{*}_{n}(x|\alpha,\lambda)\bigg)\frac{t^n}{n!}\text{,}  
\end{align}
which gives, by equating the coefficients of $\frac{t^n}{n!}$ in both sides of the above identity, the following derivative formula for the generating function $\phi^{*}_{n}(x|\alpha,\lambda)$.
\begin{theorem}
For $n\geq 0$, we have  
    $$\frac{d}{dx}\phi^{*}_{n}(x|\alpha,\lambda)=\frac{1}{1+\alpha x}\Bigg(\lambda \sum_{\ell=0}^{n}\binom{n}{\ell}\phi^{*}_{\ell}(x|\alpha,\lambda)+\phi^{*}_{n}(x|\alpha,\lambda)\Bigg).$$
\end{theorem}

By virtue of Theorem \ref{gfphi}, we have 
\begin{align}
    \sum_{n=0}^{\infty}\int_{0}^{x}\phi^{*}_{n}(y|\alpha,\lambda)dy\frac{t^n}{n!}&=\int_{0}^{x}e_{\alpha}^{\lambda e^{t}+1}(y)dy\\
    &=\Bigg[\frac{1+\alpha y}{\lambda e^t+1}e_{\alpha}^{\lambda e^t+1}(y)\Bigg]_{0}^{x}\\
    &=\frac{1}{\lambda e^t+1}\Big((1+\alpha x)e_{\alpha}^{\lambda e^t+1}(x)-1\Big)\label{integeq2}\\
    &=\frac{1+\alpha x}{2}\sum_{n=0}^{\infty}E_{n}^{(1)}(\lambda|0)\frac{t^n}{n!}\sum_{l=0}^{\infty}\phi^{*}_{l}(x,\alpha|\lambda)\frac{t^l}{l!}-\frac{1}{2}\sum_{n=0}^{\infty}E_{n}^{(1)}(\lambda|0)\frac{t^n}{n!}\\
    &=\sum_{n=0}^{\infty}\bigg(\frac{1+\alpha x}{2}\sum_{l=0}^{n}\binom{n}{l}E_{n-l}^{(1)}(\lambda|0)\phi^{*}_{l}(x,\alpha|\lambda)-\frac{1}{2}E_{n}^{(1)}(\lambda|0)\bigg)\frac{t^n}{n!}.
\end{align}
Comparing the coefficient of both sides of the above identity, we get the following integral formula for the generating function $\phi^{*}_{n}(x|\alpha,\lambda)$:
\begin{theorem}\label{integphi}
For $n\geq 1$, one has
\begin{equation}
    \int_{0}^{x}\phi^{*}_{n}(y|\alpha,\lambda)dy=\frac{1+\alpha x}{2}\sum_{l=0}^{n}\binom{n}{l}E_{n-l}^{(1)}(\lambda|0)\phi^{*}_{l}(x|\alpha,\lambda)-\frac{1}{2}E_{n}^{(1)}(\lambda|0)
\end{equation}
\end{theorem}
The following result from \cite{KNB2005, TKDSK2022} will be used in the proof of the next theorem: Let $f(x)=\sum_{n=0}^{\infty}a_nx^n$ and $g(x)=\sum_{k=0}^{\infty}b_kx^k$ be power series which are convergent on some open disk centered at the origin. Then 
\begin{equation}\label{Formpowerseries}
    \sum_{k=0}^{\infty}\frac{g^{(k)}(0)}{k!}f(k)x^k=\sum_{n=0}^{\infty}\frac{f^{(n)}(0)}{n!}\sum_{k=0}^{n}S_{2}(n,k)g^{(k)}(x)x^k.
\end{equation}
\begin{theorem}
For $n\geq 0$, on has
$$\sum_{m=0}^{\infty}y_{1,\alpha}^{*}(n,m;\lambda)f(m)x^m=\sum_{j=0}^{n}\binom{n}{j}\sum_{m=0}^{\infty}\sum_{k=0}^{m}S_{2}(m,k)\Big(\frac{x}{1+\alpha x}\Big)^kk!\frac{f^{(m)}(0)}{m!}y_{1, \alpha}^{*}(j,k;\lambda)\phi^{*}_{n-j}(x|\alpha,\lambda).$$
\end{theorem}
\begin{proof}
Consider the function $g(x)=e_{\alpha}^{\lambda e^t+1}(x)$. Then 
\begin{equation}\label{kderiv}
 g^{(k)}(x)=\Big(\frac{d}{dx}\Big)^ke_{\alpha}^{\lambda e^t+1}(x)=\frac{(\lambda e^t+1)_{k,\alpha}}{(1+\alpha x)_{k,\alpha}}e_{\alpha}^{\lambda e^t+1}(x).
\end{equation}
Using \eqref{newgSn}, \eqref{Formpowerseries} and  \eqref{kderiv}, we obtain
\begin{align}
\nonumber &\sum_{m=0}^{\infty}\frac{(\lambda e^t+1)_{m, \alpha}}{m!}f(m)x^m =\sum_{m=0}^{\infty}\frac{f^{(m)}(0)}{m!}\sum_{k=0}^{m}S_{2}(m,k)x^k\frac{(\lambda e^t+1)_{k,\alpha}}{(1+\alpha x)^k}e_{\alpha}^{\lambda e^t+1}(x)\\\nonumber
&=e_{\alpha}^{\lambda e^t+1}(x)\sum_{m=0}^{\infty}\frac{f^{(m)}(0)}{m!}\sum_{k=0}^{m}S_{2}(m,k)\Big(\frac{x}{1+\alpha x}\Big)^k k!\sum_{j=0}^{\infty}y_{1,\alpha}^{*}(j,k;\lambda)\frac{t^j}{j!}\\ \nonumber
&=\sum_{l=0}^{\infty}\phi^{*}_{l}(x|\alpha,\lambda)\frac{t^l}{l!}\sum_{m=0}^{\infty}\frac{f^{(m)}(0)}{m!}\sum_{k=0}^{m}S_{2}(m,k)\Big(\frac{x}{1+\alpha x}\Big)^k k!\sum_{j=0}^{\infty}y_{1,\alpha}^{*}(j,k;\lambda)\frac{t^j}{j!}\\ \nonumber
&=\sum_{n=0}^{\infty}\bigg(\sum_{j=0}^{n}\binom{n}{j}\sum_{m=0}^{\infty}\sum_{k=0}^{m}S_{2}(m,k)\Big(\frac{x}{1+\alpha x}\Big)^kk!\frac{f^{(m)}(0)}{m!}y_{1,\alpha}^{*}(j,k;\lambda)\phi^{*}_{n-j}(x|\alpha,\lambda)\bigg)\frac{t^n}{n!}.
\end{align}
By comparing the coefficients of both sides of the above identity, we have the desired result.
\end{proof}

\subsection{Numerical Results and Graphical Illustrations of $y_{1, \alpha}^{*}(n,k;\lambda)$}
In this subsection, similarly to the recent computational implementations given by Kucukoglu~\cite{IK2025} for $q$-combinatorial Simsek numbers and polynomials, we present numerical illustrations and graphical  plots to highlight the behavior of the generalized numbers
$y_{1,\alpha}^{*}(n,k;\lambda)$ with respect to the parameters
$n$, $k$, $\alpha$, and $\lambda$.
All computations are based on the explicit formula~\eqref{expforntsn}, adopting the convention $0^0=1$.
\medskip
We first  present explicit numerical values of $y_{1,\alpha}^{*}(n,k;\lambda)$ in tabular form for small values of
$n$ and $k$. Table~\ref{tab:alphafree_lambda1} lists the values for fixed $\lambda=1$ and
free $\alpha$, while Table~\ref{tab:alpha1_lambdafree} provides the
corresponding values for fixed $\alpha=1$ and free $\lambda$, in both cases for
$0 \leq n,k \leq 4$.
\begin{table}[h]
\centering
\setlength{\tabcolsep}{12pt} 
\renewcommand{\arraystretch}{1.5} 
\begin{tabular}{|c|ccccc|}
\hline
$n \;\backslash\; k$ & $0$ & $1$ & $2$ & $3$ & $4$ \\
\hline
$0$ &
$1$ &
$2$ &
$\alpha + 2$ &
$\frac{2\alpha^{2} + 6\alpha + 4}{3}$ &
$\frac{3\alpha^{3} + 11\alpha^{2} + 12\alpha + 4}{6}$ \\

$1$ &
$0$ &
$1$ &
$\frac{\alpha + 4}{2}$ &
$\frac{\alpha^{2} + 6\alpha + 6}{3}$ &
$\frac{3\alpha^{3} + 22\alpha^{2} + 36\alpha + 16}{12}$ \\

$2$ &
$0$ &
$1$ &
$\frac{\alpha + 6}{2}$ &
$\frac{\alpha^{2} + 9\alpha + 12}{3}$ &
$\frac{3\alpha^{3} + 33\alpha^{2} + 72\alpha + 40}{12}$ \\

$3$ &
$0$ &
$1$ &
$\frac{\alpha + 10}{2}$ &
$\frac{\alpha^{2} + 15\alpha + 27}{3}$ &
$\frac{3\alpha^{3} + 55\alpha^{2} + 162\alpha + 112}{12}$ \\

$4$ &
$0$ &
$1$ &
$\frac{\alpha + 18}{2}$ &
$\frac{\alpha^{2} + 27\alpha + 66}{3}$ &
$\frac{3\alpha^{3} + 99\alpha^{2} + 396\alpha + 340}{12}$ \\
\hline
\end{tabular}
\caption{Values of $y_{1,\alpha}^{*}(n,k;1)$ for $0\leq n, k\leq 4$.}
\label{tab:alphafree_lambda1}
\end{table}
\FloatBarrier
\begin{table}[htbp]
\centering
\setlength{\tabcolsep}{11pt} 
\renewcommand{\arraystretch}{1.7} 
\begin{tabular}{|c|ccccc|}
\hline
$n \;\backslash\; k$ & $0$ & $1$ & $2$ & $3$ & $4$ \\
\hline
$0$ &
$1$ &
$\lambda + 1$ &
$\frac{\lambda^{2} + 3\lambda + 2}{2}$ &
$\frac{\lambda^{3} + 6\lambda^{2} + 11\lambda + 6}{6}$ &
$\frac{\lambda^{4} + 10\lambda^{3} + 35\lambda^{2} + 50\lambda + 24}{24}$ \\

$1$ &
$0$ &
$\lambda$ &
$\frac{2\lambda^{2} + 3\lambda}{2}$ &
$\frac{3\lambda^{3} + 12\lambda^{2} + 11\lambda}{6}$ &
$\frac{2\lambda^{4} + 15\lambda^{3} + 35\lambda^{2} + 25\lambda}{12}$ \\

$2$ &
$0$ &
$\lambda$ &
$\frac{4\lambda^{2} + 3\lambda}{2}$ &
$\frac{9\lambda^{3} + 24\lambda^{2} + 11\lambda}{6}$ &
$\frac{8\lambda^{4} + 45\lambda^{3} + 70\lambda^{2} + 25\lambda}{12}$ \\

$3$ &
$0$ &
$\lambda$ &
$\frac{8\lambda^{2} + 3\lambda}{2}$ &
$\frac{27\lambda^{3} + 48\lambda^{2} + 11\lambda}{6}$ &
$\frac{32\lambda^{4} + 135\lambda^{3} + 140\lambda^{2} + 25\lambda}{12}$ \\

$4$ &
$0$ &
$\lambda$ &
$\frac{16\lambda^{2} + 3\lambda}{2}$ &
$\frac{81\lambda^{3} + 96\lambda^{2} + 11\lambda}{6}$ &
$\frac{128\lambda^{4} + 405\lambda^{3} + 280\lambda^{2} + 25\lambda}{12}$ \\
\hline
\end{tabular}
\caption{Values of $y_{1,1}^{*}(n,k;\lambda)$ for $0\leq n, k\leq 4$.}
\label{tab:alpha1_lambdafree}
\end{table}
\FloatBarrier
\medskip

These tables clearly illustrate the polynomial dependence of
$y_{1,\alpha}^{*}(n,k;\lambda)$ on the free parameter.
For fixed $\lambda=1$, Table~\ref{tab:alphafree_lambda1} shows that the entries
are polynomials in $\alpha$ whose degrees increase with $k$, reflecting the
role of the degeneracy parameter in deforming the classical structure.
Similarly, Table~\ref{tab:alpha1_lambdafree} demonstrates that, for fixed
$\alpha=1$, the values are polynomial expressions in $\lambda$ with coefficients
depending on $n$ and $k$.

\medskip

Moreover, the boundary cases $k=0$ and $k=1$ are consistent with the theoretical
properties of the numbers, namely
$y_{1,\alpha}^{*}(n,0;\lambda)=\delta_{n0}$ and
$y_{1,\alpha}^{*}(n,1;\lambda)=\delta_{n0}+\lambda$.

%
For completeness, we also include in Table~\ref{tab:alpha1_lambda1}
the numerical values of $y_{1,\alpha}^{*}(n,k;\lambda)$ for the representative
choice $(\alpha,\lambda)=(1,1)$ and $0 \le n,k \le 4$.
This table provides a direct numerical verification of the polynomial
expressions given in Tables~\ref{tab:alphafree_lambda1}
and~\ref{tab:alpha1_lambdafree}, and confirms the values displayed in the
corresponding graphical plots in Figures~\ref{fig:4comparison_panel} and~\ref{fig:4comparison_panel_n_fix} .

\begin{table}[htbp]
\centering
\footnotesize
\setlength{\tabcolsep}{10pt} 
\renewcommand{\arraystretch}{1.6} 
\begin{tabular}{|c|ccccc|}
\hline
$n\backslash k$ & $0$ & $1$ & $2$ & $3$ & $4$ \\ \hline
$0$ & $1$ & $2$ & $3$ & $4$ & $5$ \\
$1$ & $0$ & $1$ & $\frac{5}{2}$ & $\frac{13}{3}$ & $\frac{77}{12}$ \\
$2$ & $0$ & $1$ & $\frac{7}{2}$ & $\frac{22}{3}$ & $\frac{37}{3}$ \\
$3$ & $0$ & $1$ & $\frac{11}{2}$ & $\frac{43}{3}$ & $\frac{83}{3}$ \\
$4$ & $0$ & $1$ & $\frac{19}{2}$ & $\frac{94}{3}$ & $\frac{419}{6}$\\
\hline
\end{tabular}
\caption{Numerical values of $y_{1,1}^{*}(n,k;1)$ for $0\le n,k\le4$.}
\label{tab:alpha1_lambda1}
\end{table}
\FloatBarrier

In addition to the numerical illustrations, We also examine the variation of $y_{1,\alpha}^{*}(n,k;\lambda)$ as a function
of the order $n$ for a fixed combinatorial parameter $k$.

Figure~\ref{fig:4comparison_panel} illustrates the case $k=3$.
Panels~(a) and (c) show the values for fixed $\alpha=1$ and $\alpha=2$, respectively, with $\lambda\in\{1,2,3,4\}$, while panels~(b) and (d) illustrate the influence of the
degeneracy parameter $\alpha\in\{\frac{1}{2}, 1, 2, 3\}$ for fixed $\lambda=1$ and $\lambda=2$, respectively.
In both cases, the values increase rapidly with $n$, and larger values of
$\lambda$ or $\alpha$ lead to a pronounced growth rate, indicating a strong
dependence on these parameters.

\medskip

\begin{figure}[htbp]
\centering
\includegraphics[width=0.75\textwidth]{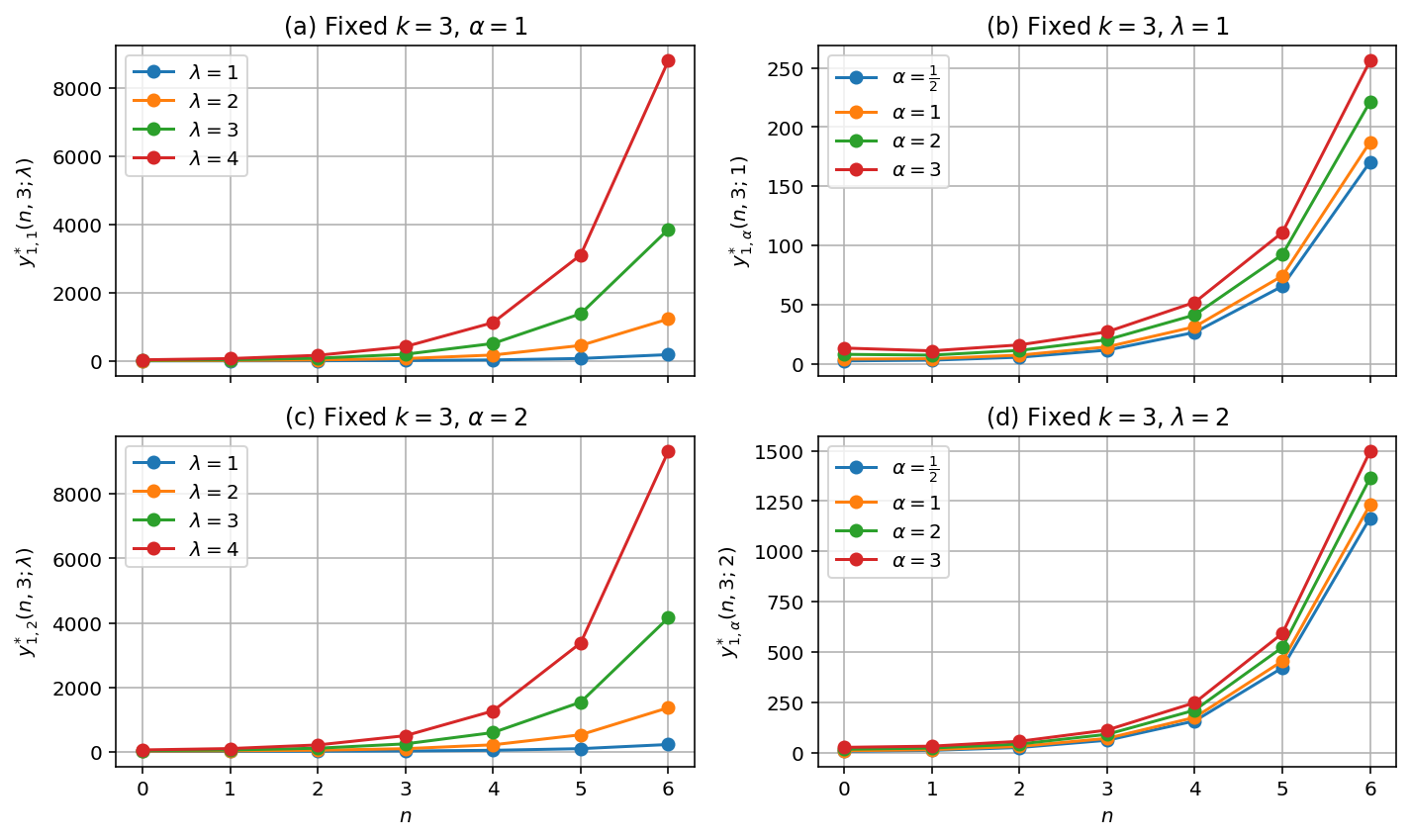}
\caption{
Numerical comparison of the degenerate Simsek numbers
$y_{1,\alpha}^{*}(n,3;\lambda)$.
Panels (a) and (c) illustrate the dependence on $\lambda$ for fixed
$\alpha$, while panels (b) and (d) show the effect of varying $\alpha$
for fixed $\lambda$.
}
\label{fig:4comparison_panel}
\end{figure}
\FloatBarrier
To complement this analysis, we next fix the order $n=3$ and study the dependence
on the parameter $k$.

Figure~\ref{fig:4comparison_panel_n_fix} presents a four-panel comparison.
Panels~(a) and~(c) correspond to fixed $\alpha=1$ and $\alpha=2$, respectively,
with $\lambda\in\{1,2,3,4\}$.
Panels~(b) and~(d) show the effect of varying $\alpha$ over $\{\frac{1}{2}, 1, 2, 3\}$ for fixed
$\lambda=1$ and $\lambda=2$, respectively.
In contrast to the growth in $n$, the dependence on $k$ is more moderate and
exhibits a structured polynomial-type behavior.
The influence of the degeneracy parameter $\alpha$ becomes more visible as $k$
increases, especially for larger values of $\lambda$.
\medskip
\begin{figure}[htbp]
\centering
\includegraphics[width=0.75\textwidth]{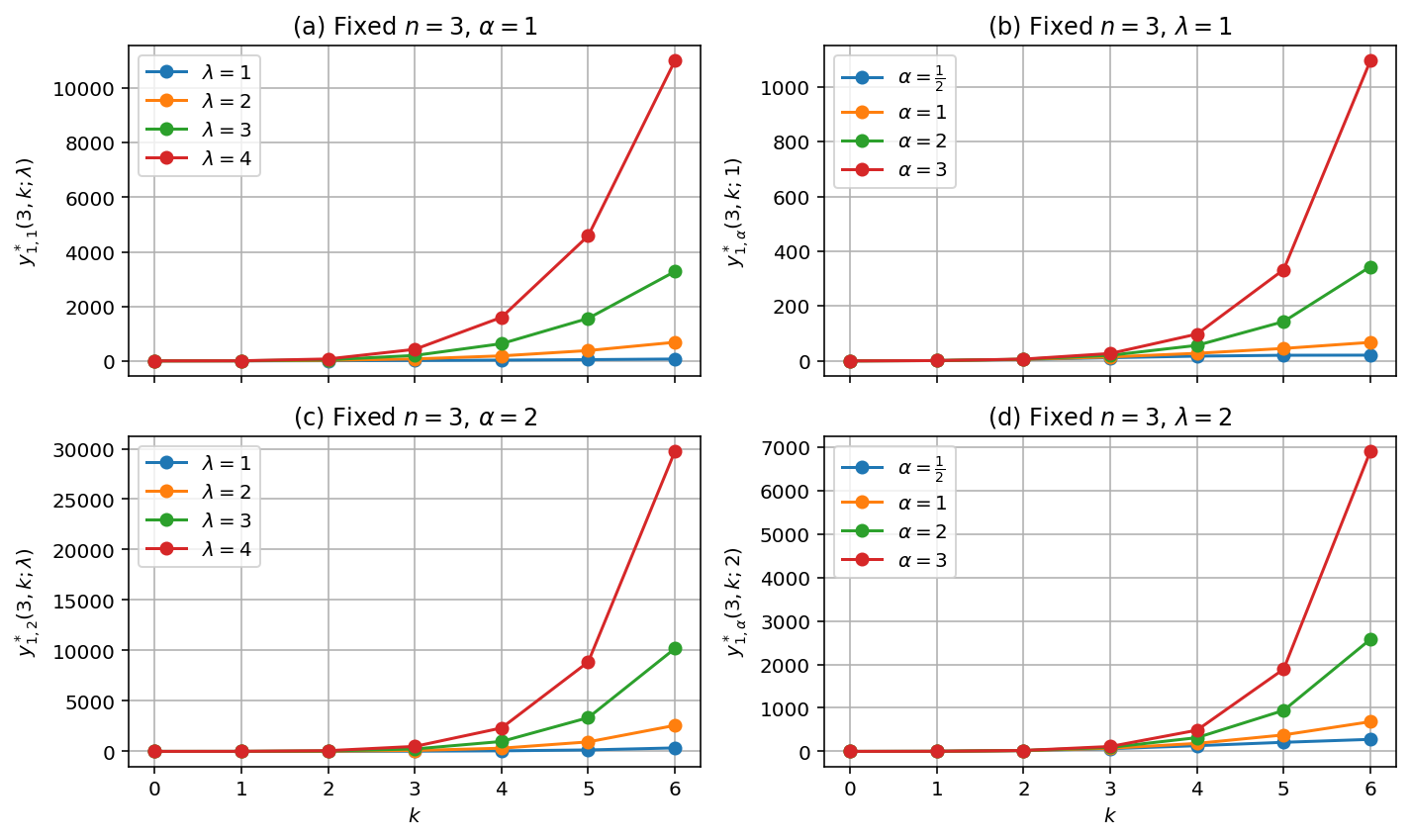}
\caption{
Numerical comparison of the degenerate Simsek numbers
$y_{1,\alpha}^{*}(3,k;\lambda)$.
Panels (a) and (c) illustrate the dependence on $\lambda$ for fixed
$\alpha$, while panels (b) and (d) show the effect of varying $\alpha$
for fixed $\lambda$.
}
\label{fig:4comparison_panel_n_fix}
\end{figure}
\FloatBarrier

Overall, these numerical and graphical experiments confirm the analytical structure of
$y_{1,\alpha}^{*}(n,k;\lambda)$ and clearly illustrate the distinct roles played
by the parameters $n$, $k$, $\alpha$, and $\lambda$.
In particular, the parameter $\lambda$ primarily controls the growth rate,
whereas $\alpha$ governs the deformation effect, which becomes more pronounced
for larger values of $k$ and $n$.

\section{Conclusions}
In the present work, we have established a new type degenerate Simsek numbers $y_{1,\alpha}^{*}(n,k;\lambda)$ with their generating function, using the degenerate falling function. These numbers are different from degenerate Simsek number studied so far. We analyzed their some properties, including derivative formula, recurrence relation, integral formula. In addition, we described the relation between these numbers and certain well-known special numbers, for instance, Stirling numbers of the first and the second kind, the degenerate first kind Apostol-Euler numbers, the degenerate Stirling numbers of the second kind, Bernoulli numbers, Simsek numbers. To complement the theoretical developments, we included numerical results and graphical illustrations that demonstrate the behavior of the introduced numbers for various parameter choices. Hence, the results of this paper have the potential to find an application in different areas of mathematics, especially dealing with discrete mathematics.

\end{document}